\documentclass[a4paper, 12pt]{amsproc}
\allowdisplaybreaks

\usepackage{fullpage}
\usepackage{amsmath, amsthm, amssymb, mathtools, mathrsfs}
\usepackage{bm}
\usepackage{hyperref}

\usepackage{chessfss}
\usepackage{mathtools}
\usepackage{tikz}
\usetikzlibrary{intersections, calc, arrows.meta}

\usepackage{pgfplots}

\usepackage{graphicx}
\usepackage{here}
\usepackage{time}
\usepackage{bbm}

\usepackage{xcolor}
\usepackage[capitalize,nameinlink,noabbrev,nosort]{cleveref}
\hypersetup{
	colorlinks=true, % false: boxed links; true: colored links
	linkcolor=brown, % color of internal links
	citecolor=brown, % color of links to bibliography
	filecolor=brown, % color of file links
	urlcolor=brown, % color of external links
}

\makeatletter
\@namedef{subjclassname@2020}{%
 \textup{2020} Mathematics Subject Classification}
\makeatother

\newtheorem{theoremcounter}{Theorem Counter}[section]

\theoremstyle{definition}

\theoremstyle{plain}
\newtheorem{lemma}[theoremcounter]{Lemma}

\newtheorem{corollary}[theoremcounter]{Corollary}

\newtheorem{theorem}[theoremcounter]{Theorem}

\numberwithin{equation}{section}

\allowdisplaybreaks

\begin{document}
% --------------------------------------------------------------------------

\title[]{The average number of Goldbach representations over multiples of $q$} 

\author{Karin Ikeda} 
\address{Joint Graduate School of Mathematics for Innovation, Kyushu University,
Motooka 744, Nishi-ku, Fukuoka 819-0395, Japan}
\email{ikeda.karin.236@s.kyushu-u.ac.jp}

\author[Suriajaya]{Ade Irma Suriajaya}
\address{Faculty of Mathematics, Kyushu University,
Motooka 744, Nishi-ku, Fukuoka 819-0395, Japan}
\email{adeirmasuriajaya@math.kyushu-u.ac.jp}

%\thanks{
%\framebox{\today \ \now}}

%11M32: Multiple Dirichlet series and zeta functions and multizeta values

% --------------------------------------------------------------------------

\maketitle

% --------------------------------------------------------------------------
\begin{abstract}
	We discuss the evaluation of the average number of Goldbach representations for integers which are multiples of $q$ introduced by Granville. We improve an estimate given by Granville under the generalized Riemann hypothesis.

 \end{abstract}
% --------------------------------------------------------------------------

% ------------------------1------Introduction------------------------------------------------------
\section{Introduction}
This paper describes a relationship between the average number of Goldbach representation 
$$
G(N):=\sum_{n\le N}\psi_2(n)
$$
and its variant
$$
G_q(N):=\sum_{\substack{n\le N\\q|n}}\psi_2(n)\qquad(1\le q\le N).
$$
Here
$$
\psi_2(n):=\sum_{m+m'=n}\Lambda(m)\Lambda(m')
$$
with $\Lambda(n)$ the von Mangoldt function, which counts the number of Goldbach representations of $n$ in primes and prime powers. Here and in the following, all sums are over positive integers unless otherwise indicated, and $N\ge 4$ so that $G(N)>0$.
Fujii~\cite{F} proved the following explicit formula for $G(N)$.
\begin{theorem}{\cite{F}}\label{Fujii}
Assuming the Riemann hypothesis, we have
$$
G(N)=\frac{N^2}{2}-2\sum_{\rho}\frac{N^{\rho+1}}{\rho(\rho+1)}+\mathcal{O}\left((N\log N)^{\frac{4}{3}}\right),
$$
where the sum is over the nontrivial zeros $\rho$ of the Riemann zeta function, and the Riemann hypothesis states that ${\rm Re}(\rho)=1/2$.
\end{theorem}

Using classical results, one immediately sees that the Riemann hypothesis implies that the second main term in Theorem \ref{Fujii} is bounded above by $N^{3/2}$.
In fact, it is known that
$$
G(N)-\frac{N^2}{2}=\mathcal{O}(N^{\frac{3}{2}+\epsilon})\qquad(\forall\epsilon>0)
$$
is equivalent to the Riemann hypothesis, see \cite{G1,G2}, \cite{BhowRuzsa2018} and \cite[Theorem 1 (2)]{BHMS}.
In \cite{BhowRuzsa2018} and \cite{BHMS}, and in addition \cite{BCSS}, more general equivalences between bounds for the right-hand side above and zero-free regions of the Riemann zeta function are obtained.

Fujii's \Cref{Fujii} was improved by Bhowmik and Schlage-Puchta~\cite{BS} who showed that $\mathcal{O}\left((N\log N)^{\frac{4}{3}}\right)$ can be replaced by $\mathcal{O}(N\log^5\!N)$, and further by Languasco and Zaccagnini~\cite{LZ} as follows.

\begin{theorem}{\cite{LZ}}\label{now}
Assuming the Riemann hypothesis, we have
$$
G(N)=\frac{N^2}{2}-2\sum_{\rho}\frac{N^{\rho+1}}{\rho(\rho+1)}+\mathcal{O}(N\log^3\!N),
$$
where the sum is over the nontrivial zeros of the Riemann zeta function.
\end{theorem}

Granville~\cite{G1,G2} introduced the function $G_q(N)$ defined by
$$
G_q(N):=\sum_{\substack{n\le N\\q|n}}\psi_2(n),\qquad(2\le q\le N).
$$
When $q=1$, we set $G_1(N)=G(N)$. Granville stated the formula, for any $\epsilon >0$, 
\begin{align}\label{previous}
G_q(N)=\frac{G(N)}{\phi(q)}+\mathcal{O}(N^{1+\epsilon}),
\end{align}
where $\phi(q)$ is Euler's totient function, but the proof of this had to wait until the work of Bhowmik and Schlage-Puchta~\cite{BS} or Languasco and Zaccagnini~\cite{LZ}. Either of these papers can be used to prove the stronger result that
\[G_q(N)=\frac{G(N)}{\phi(q)}+\mathcal{O}(N \log^C\!N)\]
for some fixed constant $C$. There are now a variety of methods for proving this result, see for example \cite{BHMS} where they considered much more general problems.
In this paper, we show that $C=3$ using the same basic method as Languasco and Zaccagnini~\cite{LZ} with some additional refinements from \cite{GS2}. An outline of this proof appeared in Goldston and Suriajaya's work \cite{GS1}, in which they discussed a smooth version of the following theorem.
\begin{theorem}\label{main}
Assume the generalized Riemann hypothesis for Dirichlet $L$-functions $L(s,\chi)$ associated with characters $\chi$ (${\rm mod}~q$).
For $2\le q\le N$, we have 
$$
G_q(N)=\frac{G(N)}{\phi(q)}+\mathcal{O}(N\log^3\!N).
$$
\end{theorem} 
Furthermore, by using \cref{now} and \cref{main}, we easily obtain the following corollary.
\begin{corollary}\label{cor}
Under the same assumption of \cref{main}, 
we have 
$$
G_q(N)=\frac{1}{\phi(q)}\left(\frac{N^2}{2}-2\sum_{\rho}\frac{N^{\rho+1}}{\rho(\rho+1)}\right)+\mathcal{O}(N\log^3\!N),
$$
where the sum is over the nontrivial zeros of the Riemann zeta function.
\end{corollary}

An analogous problem was also considered by Nguyen in her recent work \cite{ThiNgu1}.
She considered a general case with prime powers from arbitrary arithmetic progressions using
$$
\sum_{n\le X}\sum_{\substack{m+m'=n \\ m\equiv a \,({\rm mod}\, q) \\ m'\equiv a' \,({\rm mod}\, q')}}\Lambda(m)\Lambda(m')
$$
for positive integers $a, a', q, q'$ satisfying $1\le a<q$, $(a,q)=1$ and $1\le a'<q'$, $(a',q')=1$.
Our $G_q(N)$ is the case $q=q'$ and includes the case $a=a'=q$ which is not covered in \cite{ThiNgu1}.
Nguyen \cite{ThiNgu1} considered also cases when not assuming the generalized Riemann hypothesis analogous to \cite{BHMS}, in addition to a \cite{BS}-type omega result and a Ces\`aro weighted average as in \cite{Lang-Zac2015,CGZ21}.
Finally, we remark that all the $\mathcal{O}$-constants mentioned above are absolute. In the following section, we retain the same notation, unless otherwise specified.

% --------------2------------Proofs-------------------------------------------------------
\section{Proof of \cref{main}}
%-------------------2.1------
%\subsection{Proof of \cref{main}}
Our method follows closely that of \cite{GS1} where this problem was introduced and \cite{GS2} by Goldston and the second author.
Following earlier works \cite{MV,LZ,GS2,GS1}, we use power series generating functions. This approach originates from Hardy and Littlewood's circle method \cite{HL}.
As in \cite{GS1}, we define for $N\ge 4$ a smooth version of $G_q(N)$ as
$$
F_q(z)=\sum_{\substack{n\\ q|n}}\psi_2(n)z^n,
$$
with $z=re(\alpha)$ and $r=e^{-1/N}$, where $e(\alpha)=e^{2\pi i \alpha}$.
We first show that we can represent $G_q(N)$ by using $F_q(z)$. Setting
$$
I_N(z):=\sum_{n\le N}z^n=z\left(\frac{1-z^N}{1-z}\right),
$$
we have for $1\le q \le N$, 
\begin{align}\label{goma1}
\begin{split}
\int_{0}^{1}F_q(z)I_N\left(\frac{1}{z}\right)d\alpha
&=\int_{0}^{1}\sum_{\substack{n\\q|n}}\psi_2(n)z^n\sum_{n'\le N}\left(\frac{1}{z}\right)^{n'}d\alpha\\
&=\int_{0}^{1}\sum_{\substack{n\\q|n}}\psi_2(n)r^n(e(\alpha))^n\sum_{n'\le N}r^{-n'}(e(\alpha))^{-n'}d\alpha\\
&=\sum_{\substack{n\\q|n}}\psi_2(n)r^n\left(\sum_{n'\le N}r^{-n'}\int_{0}^{1}e(\alpha(n-n'))d\alpha\right)\\
&=\sum_{\substack{n\le N\\q|n}}\psi_2(n)\\
&=G_q(N),
\end{split}
\end{align}
see also \cite[Equation (40)]{GS2}.
In particular, when $q=1$, we get
\begin{align}\label{goma2}
\begin{split}
G(N)=G_1(N)
&=\int_{0}^{1}F_1(N)I_N\left(\frac{1}{z}\right)d\alpha=\int_{0}^{1}\Psi^2(z)I_N\left(\frac{1}{z}\right)d\alpha,
\end{split}
\end{align}
where
\begin{equation}\label{Psi(z)}
\Psi(z)=\sum_{n}\Lambda(n)z^n.
\end{equation}

Next, for a Dirichlet character $\chi$ (mod $q$), we let
$$
\Psi(z,\chi)=\sum_{n}\chi(n)\Lambda(n)z^n.
$$
As shown in \cite{GS1}, $F_q(z)$ can be approximated using a character weighted average of $\Psi(z,\chi)$ as follows.
\begin{lemma}{\cite[Lemma~2.1]{GS1}}\label{GSthm1}
For $N\ge 4$ and $q\ge 2$,
$$
F_q(z)=\frac{1}{\phi(q)}\sum_{\chi ({\rm mod}\,q)}\chi(-1)\Psi(z,\chi)\Psi(z,\bar{\chi})+\mathcal{O}((\log N \log q)^2).
$$
For $q=1$, $F_1(z)=\Psi(z)^2$.
\end{lemma}
\noindent
The next lemma shows that with respect to the principal character $\chi_0$ (mod $q$) we can approximate $\Psi(z,\chi_0)$ using $\Psi(z)$ defined in \eqref{Psi(z)}.
\begin{lemma}{\cite[(17) of Lemma~2.2]{GS1}}\label{GSthm2}
Let $q\ge 2$ and $\chi_0$ be the principal character (mod $q$). Then we have
$$
\Psi(z,\chi_0)=\sum_{\substack{n\\(n,q)=1}}\Lambda(n) z^n=\Psi(z)+\mathcal{O}(\log N\log q).
$$
\end{lemma}
\noindent
We remark that \cite[Lemma~2.2]{GS1} also includes the case of general characters $\chi$, see also \cite[Lemma 2.1]{S}. The proof can be found in \cite[Lemma 2]{LZ2} or \cite[Lemmas 1 to 4]{HL}.

We use the short-hand notation
\[ \sum_{\chi} \qquad \text{and} \qquad \sum_{\chi\neq \chi_0} \]
to denote respectively a sum over all characters modulo $q$ and a sum over all characters modulo $q$ except the principal character. Now, applying \cref{GSthm1} and \cref{GSthm2}, we divide the integral on the left-hand side of \eqref{goma1} into three parts
\begin{align*}
G_q(N)
&=\int_{0}^{1}\biggl(\frac{1}{\phi(q)}\sum_{\chi}\chi(-1)\Psi(z,\chi)\Psi(z,\bar{\chi})+\mathcal{O}((\log N \log q)^2)\biggr) I_N\left(\frac{1}{z}\right)d\alpha\\
&=\mathcal{I}_1+\mathcal{I}_2+ \mathcal{O}(\mathcal{I}_3),
\end{align*}
where
\begin{align*}
&\mathcal{I}_1=\int_{0}^{1}\frac{1}{\phi(q)}\chi_0(-1)\Psi(z,\chi_0)\Psi(z,\bar{\chi_0})I_N\left(\frac{1}{z}\right)d\alpha,\\
&\mathcal{I}_2=\int_{0}^{1}\left(\frac{1}{\phi(q)}\sum_{\chi\ne\chi_0}\chi(-1)\Psi(z,\chi)\Psi(z,\bar{\chi})\right)I_N\left(\frac{1}{z}\right)d\alpha,
\end{align*}
and
$$
\mathcal{I}_3 = (\log N \log q)^2\int_{0}^{1}\left|I_N\left(\frac{1}{z}\right)\right|d\alpha.
$$
The term $\mathcal{I}_2$ is an error term that is difficult to estimate, so we handle the easy terms $\mathcal{I}_1$ and $\mathcal{I}_3$ immediately.\\

%----------Part1-----------------------------
\noindent\textbf{Evaluation of} \bm{$\mathcal{I}_1$} \textbf{and} \bm{$\mathcal{I}_3$}\textbf{.}

We first need to estimate
$$
\int_{0}^{1}\left|I_N\left(\frac{1}{z}\right)\right|d\alpha.
$$
%For $|\alpha|<1/2$, 
We recall that $|z|=r=e^{-1/N}\le1$, hence we can trivially bound $I_N(1/z)$ as
\begin{equation}\label{tuika2}
\left|I_N\left(\frac{1}{z}\right)\right|
=\left|\sum_{n\le N}\left(\frac{1}{z}\right)^n\right|
=\left|\frac{1}{z^N}(1+z+\cdots+z^{N-1})\right|
\le eN.
\end{equation}
If $|\alpha|\le 1/2$, then
\[\begin{split} |z-1| &= \sqrt{(r\cos(2\pi \alpha)-1)^2 + (r\sin(2\pi \alpha))^2} = \sqrt{ r^2 +1 - 2r\cos(2\pi \alpha)} \\ &
= \sqrt{(r-1)^2 + 2r(1-\cos(2\pi \alpha))}= \sqrt{(r-1)^2 + 4r(\sin\pi \alpha)^2}\\&
\ge 2\sqrt{r}\sin(\pi|\alpha|) \ge 4 \sqrt{r} |\alpha|,
\end{split}\]
where the last inequality 
%uses the secant line under the curve $\sin \pi \alpha$ between $\alpha=0$ and $\alpha = 1/2$.
uses the fact that $\pi\sin\theta-2\theta\ge0$ for $0\le\theta\le\pi/2$.
Thus we have, for $0<|\alpha|\le 1/2$,
\begin{equation}\label{goma5}
\begin{aligned}
\left|I_N\left(\frac{1}{z}\right)\right|
= \left|\frac{1}{z}\frac{1-\frac{1}{z^N}}{1-\frac{1}{z}}\right|
= \left|\frac{1-z^{-N}}{z-1}\right|
\le \frac{1+|z|^{-N}}{4\sqrt{r}|\alpha|} 
\le \frac{(1+e)\sqrt{e}}{4|\alpha|} 
< \frac{e}{|\alpha|}.
\end{aligned}
\end{equation}
By \eqref{tuika2} and \eqref{goma5}, we conclude for $|\alpha|\le 1/2$ that 
\begin{equation} \label{goma6} \left|I_N\left(\frac{1}{z}\right)\right|\le e\min\left\{N,\frac{1}{|\alpha|}\right\}\end{equation}
Therefore
\begin{align*}
\int_{0}^{1}\left|I_N\left(\frac{1}{z}\right)\right|d\alpha
%&=\int_{-\frac{1}{2}}^{\frac{1}{2}}\left|I_N\left(\frac{1}{z}\right)\right|d\alpha\\
&=\int_{-\frac{1}{2}}^{\frac{1}{2}}\left|I_N\left(\frac{1}{z}\right)\right|d\alpha\ll\int_{-\frac{1}{2}}^{\frac{1}{2}}\min\left\{N,\frac{1}{|\alpha|}\right\}d\alpha.
\end{align*}
For $N\ge 2$ we have \begin{align*}
\int_{-\frac{1}{2}}^{\frac{1}{2}}\min\left\{N,\frac{1}{|\alpha|}\right\}d\alpha
&=2\left(\int_{0}^{\frac{1}{N}}Nd\alpha+\int_{\frac{1}{N}}^{\frac{1}{2}}\frac{d\alpha}{\alpha}\right)\\
&=2\left(1+\log \frac{N}{2}\right)\ll\log N,
\end{align*}
and thus
\begin{align}\label{goma7}
\int_{0}^{1}\left|I_N\left(\frac{1}{z}\right)\right|d\alpha=\mathcal{O}(\log N).
\end{align}

From this estimate we see immediately that $\mathcal{I}_3 \ll \log^5\!N$.
Applying \cref{GSthm2} and \eqref{goma2}, we obtain 
\begin{align}\nonumber
\mathcal{I}_1
&=\frac{1}{\phi(q)}\int_{0}^{1}\left(\Psi(z)+\mathcal{O}(\log N\log q)\right)^2I_N\left(\frac{1}{z}\right)d\alpha\\\nonumber
&=\frac{1}{\phi(q)}\int_{0}^{1}\Psi^2(z)I_N\left(\frac{1}{z}\right)d\alpha+\mathcal{O}\left(\frac{\log N\log q}{\phi(q)}\int_{0}^{1}\left|\Psi(z)I_N\left(\frac{1}{z}\right)\right|d\alpha\right) \\
&\qquad +\mathcal{O}\left(\frac{\mathcal{I}_3}{\phi(q)}\right) \notag\\\label{goma3}
\begin{split}
&=\frac{G(N)}{\phi(q)}+\mathcal{O}\left(\frac{\log N\log q}{\phi(q)}\int_{0}^{1}\left|\Psi(z)I_N\left(\frac{1}{z}\right)\right|d\alpha\right)+\mathcal{O}\left(\frac{\log^5\!N}{\phi(q)}\right).
\end{split}
\end{align}
Next by partial summation, we can easily show that
\begin{align*}
|\Psi(z)|
&\le \sum_{n}\Lambda(n)e^{-\frac{n}{N}}
\ll \frac1N\int_1^\infty \psi(t)e^{-\frac{t}N} dt
%=\frac{1}{N}\int_{1}^{\infty}\sum_{n\le t}\Lambda(n)e^{-\frac{t}{N}}dt\\
\ll\frac{1}{N}\int_{1}^{\infty}te^{-\frac{t}{N}}dt
\ll N,
\end{align*}
where $\psi(x):=\sum_{n\le x}\Lambda(n)$ and we have used the weaker estimate $\psi(x)\ll x$ than the Prime Number Theorem $\psi(x)\sim x$.
Thus we have that the second term of \eqref{goma3} is 
\begin{equation}\label{goma4}
\ll \frac{N\log^2\!N}{\phi(q)}\int_{0}^{1}\left|I_N\left(\frac{1}{z}\right)\right|d\alpha\\
\ll \frac{N\log^3\!N}{\phi(q)},
\end{equation}
and therefore
\begin{align}\label{goma8}
\mathcal{I}_1 + \mathcal{O}( \mathcal{I}_3) 
=\frac{G(N)}{\phi(q)}+\mathcal{O}\left(\frac{N\log^3\!N}{\phi(q)}\right) + \mathcal{O}(\log^5\!N).
\end{align}
Thus
\begin{equation} \label{OnlyI2left} G_q(N) = \frac{G(N)}{\phi(q)}+ \mathcal{I}_2 + \mathcal{O}\left(\frac{N\log^3\!N}{\phi(q)}\right) + \mathcal{O}(\log^5\!N). \end{equation}
%-----------Part2-----------------------------
\noindent\textbf{Evaluation of} \bm{$\mathcal{I}_2$}\textbf{.}
Since 
\begin{align}\nonumber
|\mathcal{I}_2|
&\le\int_{0}^{1}\bigg|\frac{1}{\phi(q)}\sum_{\chi\ne\chi_0}\chi(-1)\Psi(z,\chi)\Psi(z,\bar{\chi})\bigg|\left|I_N\left(\frac{1}{z}\right)\right|d\alpha\\\nonumber
&\le \frac{1}{\phi(q)}\sum_{\chi\ne\chi_0}\int_{0}^{1}|\Psi(z,\chi)|^2\left|I_N\left(\frac{1}{z}\right)\right|d\alpha\\\nonumber
&\le \underset{\chi\ne\chi_0}{\text{max}}\int_{0}^{1}|\Psi(z,\chi)|^2\left|I_N\left(\frac{1}{z}\right)\right|d\alpha, 
\end{align}
we have by \eqref{goma6}
\begin{align}\notag
\mathcal{I}_2
&\ll \underset{\chi\ne\chi_0}{\text{max}}\int_{0}^{\frac{1}{2}}|\Psi(z,\chi)|^2\min\left\{N,\frac{1}{\alpha}\right\}d\alpha
\\\notag & \ll \underset{\chi\ne\chi_0}{\text{max}}\left(N\int_0^{\frac1{N}}|\Psi(z,\chi)|^2 \,d\alpha + \sum_{0\le k<\log_2\!N}\frac{N}{2^k}\int_{\frac{2^k}{N}}^{\frac{2^{k+1}}{N}}|\Psi(z,\chi)|^2 \,d\alpha \right) \\ &
\ll N \sum_{0\le k<\log_2\!N}\frac{1}{2^k}\left(\underset{\chi\ne\chi_0}{\rm max}\int_0^{\frac{2^{k+1}}{N}}|\Psi(z,\chi)|^2 \,d\alpha \right). \label{goma12}
\end{align}

%Letting $0<h\le 1$, we define
%\begin{align}\label{goma12}
%I(N, h) = I(N,h,\chi):= \int_{0}^{\frac1{2h}}|\Psi(z,\chi)|^2d\alpha,\quad\text{for}\quad \chi\neq\chi_0, \end{align}
%and have obtained
%\begin{equation}\label{calI2step1} \mathcal{I}_2\ll N \sum_{0\le k<\log_2\!N}\frac{1}{2^k}~\underset{\chi\ne\chi_0}{\max}\,I\bigl(N,\frac{N}{2^{k+2}},\chi\bigr). \end{equation}

The key tool in estimating the integral in \eqref{goma12}, and thus $\mathcal{I}_2$, is Gallagher's lemma~\cite[Lemma~1.9]{M}; one form of which can be stated as follows. 
\begin{lemma}[Gallagher]\label{Gallagher} For any sequence of complex numbers $\{c_n\}$, $n \in \mathbb{Z}$, and
\[ S(\alpha) = \sum_{n=-\infty}^{\infty} c_ne(n\alpha), \qquad \text{where}\qquad \sum_{n=-\infty}^{\infty} |c_n| <\infty,\]
we have for any $h>0$ that
\[\int_{\frac{-1}{2h}}^{\frac{1}{2h}}|S(\alpha)|^2\, d\alpha \ll \frac{1}{h^2}\int_{-\infty}^{\infty} \bigg|\sum_{x<n\le x+h}c_n \bigg|^2 \, dx.\]
\end{lemma}
%
%\begin{theorem}{\cite[Lemma~1.9]{M}}\label{Mthm}
%For any $\epsilon>0$, if $\delta$ and $T$ are positive real numbers for which $\delta T\le 1-\epsilon$, and let 
%$$
%S(t):=\sum_{\mu}c(\mu)e(\mu t)
%$$
%be for $c(\mu)$ such that
%$$
%\sum_{\mu}|c(\mu)|<\infty.
%$$
%Then the following holds
%$$
%\int_{-T}^{T}|S(t)|^2dt\ll_{\epsilon}\int_{-\infty}^{\infty}|c_{\delta}(x)|^2dx
%$$
%where, 
%$$
%c_{\delta}(x)=\delta^{-1}\sum_{|\mu-x|<\frac{\delta}{2}}c(\mu).
%$$
%\end{theorem}
\noindent
Applying this lemma, we have
\begin{align}\nonumber
\int_{0}^{\frac{1}{2h}}|\Psi(z,\chi)|^2\, d\alpha
&= \int_{0}^{\frac{1}{2h}}\left|\sum_{n=1}^\infty\chi(n)\Lambda(n)r^ne(n\alpha)\right|^2d\alpha\\\nonumber
&\ll\frac{1}{h^2}\int_{-\infty}^{\infty}\left|\sum_{x<n\le x+h}\chi(n)\Lambda(n)e^{-\frac{n}{N}}\right|^2dx\\\nonumber
&=\frac{1}{h^2}\int_{-h}^{0}\left|\sum_{n\le x+h}\chi(n)\Lambda(n)e^{-\frac{n}{N}}\right|^2dx+\frac{1}{h^2}\int_{0}^{\infty}\left|\sum_{x<n\le x+h}\chi(n)\Lambda(n)e^{-\frac{n}{N}}\right|^2dx\\\nonumber
&=\frac{1}{h^2}\int_{0}^{h}\left|\sum_{n\le x}\chi(n)\Lambda(n)e^{-\frac{n}{N}}\right|^2dx+\frac{1}{h^2}\int_{0}^{\infty}\left|\sum_{x<n\le x+h}\chi(n)\Lambda(n)e^{-\frac{n}{N}}\right|^2dx \\\label{goma13}
&=:\frac1{h^2}\bigl(I_1(N,h)+I_2(N,h)\bigr).
\end{align}

We introduce the counting function 
\begin{equation*}%\label{psi-chi}
\psi(x,\chi):=\sum_{n\le x}\chi(n)\Lambda(n),
\end{equation*}
which is simply the Chebyshev function twisted by a Dirichlet character $\chi$.
Next, we define
\begin{align}
&J_1(X) = J_1(X,\chi):=\int_{0}^{X}\left|\psi(x,\chi)\right|^2dx, \label{J1}\\
&J_2(X,h)= J_2(X,h,\chi) :=\int_{0}^{X}\left|\psi(x+h,\chi) -\psi(x,\chi)\right|^2dx. \label{J2}
\end{align} 
In the rest of the paper we will frequently use (without comment) the inequality, for $a,b\in \Bbb{C}$, 
\begin{equation}\label{ab2}
|a+b|^2 \le (|a|+|b|)^2 \le (|a|+|b|)^2+(|a|-|b|)^2 = 2(|a|^2+|b|^2).
\end{equation}
\begin{lemma}\label{GVlemma} Assuming GRH and $X\ge 1$. Then for any Dirichlet character $\chi\neq \chi_0$ modulo $q$, we have 
\begin{equation}\label{J1bound}
J_1(X)\ll X^2\log^2(2q)
\end{equation}
and, for $0\le h\le X$,
\begin{equation}\label{J2bound}
J_2(X,h) \ll(h+1)X\log^2\left(\frac{3qX}{h+1}\right).
\end{equation}
\end{lemma}
\begin{proof} This is proved in \cite[Lemma~2]{GV} for $\chi$ primitive, so we only need to deal with the case of $\chi$ imprimitive. If $\chi$ is an imprimitive character modulo $q$, then it is induced by a primitive character $\chi^\ast$ modulo $q^\ast$, where $q^\ast | q$. Thus, since we know \eqref{J1bound} and \eqref{J2bound} hold for primitive characters, we have
\[ J_1(X,\chi^\ast )\ll X^2\log^2 (2 q^\ast)\ll X^2\log^2 (2 q) \]
and 
\[ J_2(X,h,\chi^\ast )\ll (h+1)X\log^2 \left(\frac{3q^\ast X}{h+1}\right)\ll (h+1)X\log^2 \left(\frac{3qX}{h+1}\right). \]
Now, following \cite[p. 119]{D},
\begin{equation}\label{psiImprim}|\psi(x,\chi)-\psi(x,\chi^\ast) |\le \sum_{\substack{n\le x\\ (n,q)>1}}\Lambda(n)= \sum_{p|q}\sum_{\substack{m\\ p^m\le x}}\log p \ll (\log (x+2))(\log 2q),\end{equation}
and therefore
\[ \begin{split} J_1(X,\chi) & = \int_{0}^{X}\left|\psi(x,\chi^\ast) + (\psi(x,\chi)- \psi(x,\chi^\ast))\right|^2dx\\ 
&\ll \int_{0}^{X}\left|\psi(x,\chi^\ast)\right|^2dx + \int_{0}^{X}\left|(\log (x+2))( \log 2q)\right|^2dx \\
&\ll X^2\log^2(2q) + X\log^2(X+2)\log^2(2q) \\
&\ll X^2\log^2(2q),
\end{split} \]
which proves \eqref{J1bound} for imprimitive characters.

Next, for $J_2(X,h,\chi)$, we first note that
\[\begin{split} J_2(h,h,\chi) = &\int_{0}^{h}\left|\psi(x+h,\chi) -\psi(x,\chi)\right|^2dx \ll \int_0^{2h} |\psi(x,\chi)|^2dx \\&\ll (h+1)^2\log^2(2q)\ll (h+1)X\log^2(2q)\end{split}\]
on using \eqref{J1bound} which we just proved for all $\chi\neq \chi_0$, and therefore the piece $J_2(h,h,\chi)$ of $J_2(X,h,\chi)$ satisfies \eqref{J2bound}. Letting $\psi_h(x,\chi) := \psi(x+h,\chi)-\psi(x,\chi)$, it remains to prove that 
\[J_2(X,h,\chi)-J_2(h,h,\chi) = \int_h^X |\psi_h(x,\chi)|^2dx\]
satisfies \eqref{J2bound}.
We argue as in \eqref{psiImprim}, making use of the lower bound $x\ge h$ not available there. Thus
\begin{align*}
\Big|\psi_h(x,\chi)-\psi_h(x,\chi^\ast)\Big|
&= \Big|\sum_{x< n \le x+h}(\chi(n)-\chi^\ast(n))\Lambda(n)\Big| \\
&\le \sum_{\substack{x< n \le x+h\\ (n,q)>1}}\Lambda(n)=\sum_{p|q}\sum_{\substack{m\\ x<p^m\le x+h}}\log p .
\end{align*}
The condition $x<p^m\le x+h$ is equivalent to
\[ \frac{\log x}{\log p}<m\le \frac{\log x}{\log p} + \frac{\log (1+\frac{h}{x})}{\log p}. \]
Since $x\ge h$, we have 
\[ \frac{\log (1+\frac{h}{x})}{\log p}\le \frac{\log 2}{\log p} \le 1,\]
and therefore there is at most 1 solution for $m$ in the interval above. Hence 
$\big|\psi_h(x,\chi) -\psi_h(x,\chi^\ast)\big|\le \log q$ and 
\begin{align*} J_2(X,h,\chi)-J_2(h,h,\chi) &= \int_h^X \left|\psi_h(x,\chi^\ast) + O(\log 2q)\right|^2dx\\
& \ll \int_h^X \left|\psi_h(x,\chi^\ast)\right|^2dx + X\log^2(2q) \\&
\ll (h+1)X\log^2 \left(\frac{3q X}{h+1}\right).
\end{align*}
\end{proof}

We now are ready to estimate $I_1(N,h)$ and $I_2(N,h)$ using partial summation and \cref{GVlemma}. Starting with the counting function $\psi(x,\chi)$ we have 
\begin{align*}
\sum_{n\le x}\chi(n)\Lambda(n)e^{-\frac{n}{N}}
= \int_0^x e^{-\frac{u}{N}}d\psi(u,\chi) = \psi(x,\chi)e^{-\frac{x}{N}} + \frac{1}{N}\int_{0}^{x}\psi(u,\chi)e^{-\frac{u}{N}}du.
\end{align*}
Thus
\begin{align*}
I_1(N,h)
&= \int_{0}^{h}\left|\sum_{n\le x}\chi(n)\Lambda(n)e^{-\frac{n}{N}}\right|^2dx \\
&= \int_{0}^{h}\left|\psi(x,\chi)e^{-\frac{x}{N}} + \frac{1}{N}\int_{0}^{x}\psi(u,\chi)e^{-\frac{u}{N}}du \right|^2dx \\
&\le 2\int_{0}^{h}|\psi(x,\chi)|^2e^{-\frac{2x}{N}}dx + \frac{2}{N^2}\int_{0}^{h}\left(\int_{0}^{x}|\psi(u,\chi)|e^{-\frac{u}{N}}du\right)^2dx\\&\le 2\int_{0}^{h}|\psi(x,\chi)|^2 dx + \frac2{N^2}\int_{0}^{h}\left(\int_{0}^{x}|\psi(u,\chi)|^2du\right) \left(\int_{0}^{x}e^{-\frac{2u}{N}}du\right)dx,
\end{align*}
where we used the Cauchy-Schwarz inequality to obtain the last line. Thus
\begin{align}
I_1(N,h) &\le 2J_1(h) + \frac1N \int_{0}^{h}\int_{0}^{x}|\psi(u,\chi)|^2\,du\,dx \notag\\
&\le 2J_1(h) + \frac{h}N \int_{0}^{h}|\psi(u,\chi)|^2\,du
= \left(2+\frac{h}N\right)J_1(h)
%&\le \frac{2 h}{N^2} J_1(h) \int_{0}^\infty e^{-\frac{2u}{N}}du = \frac{h}{N}J_1(h)
\le 3J_1(h). \label{goma15}
\end{align}

Proceeding in the same way for $I_2(N,h)$, we have by partial summation
\begin{align*}
\sum_{x<n\le x+h}\chi(n)\Lambda(n)e^{-\frac{n}{N}} 
&= \int_x^{x+h} e^{-\frac{u}{N}}d\psi(u,\chi) \\&
= \psi(x+h,\chi)e^{-\frac{x+h}{N}} - \psi(x,\chi)e^{-\frac{x}{N}} + \frac{1}{N}\int_{x}^{x+h}\psi(u,\chi)e^{-\frac{u}{N}}du \\&
=\Big(\psi(x+h,\chi)-\psi(x,\chi)\Big)e^{-\frac{x}{N}} + (\psi(x+h,\chi)\left(e^{-\frac{x+h}{N}}-e^{-\frac{x}{N}}\right) \\&
\qquad+ \frac{1}{N}\int_{x}^{x+h}\psi(u,\chi)e^{-\frac{u}{N}}du\\&
=\Big(\psi(x+h,\chi)-\psi(x,\chi)\Big)e^{-\frac{x}{N}} + \mathcal{O}\left(\frac{h}{N} |\psi(x+h,\chi)|e^{-\frac{x}{N}} \right) \\&
\qquad+ \mathcal{O}\left(\frac{1}{N}\int_{x}^{x+h}|\psi(u,\chi)|e^{-\frac{u}{N}}du\right),
\end{align*}
where in the first error term we used the estimate 
\[ \left|e^{-\frac{x+h}{N}}-e^{-\frac{x}{N}}\right| = \left|-\int_{\frac{x}{N}}^{\frac{x}{N}+\frac{h}{N}}e^{-u}du\right| \ll \frac{h}{N}e^{-\frac{x}{N}}.\]
Substituting and using \eqref{ab2} repeatedly, we have 
\[
\begin{split}
I_2(N,h) &=\int_{0}^{\infty}\left|\sum_{x<n\le x+h}\chi(n)\Lambda(n)e^{-\frac{n}{N}}\right|^2dx \\
&\ll \int_{0}^{\infty}|\psi(x+h,\chi)-\psi(x,\chi)|^2e^{-\frac{2x}{N}}dx+\frac{h^2}{N^2}\int_{0}^{\infty}|\psi(x+h,\chi)|^2e^{-\frac{2x}{N}}dx\\
&\qquad\qquad+ \frac{1}{N^2}\int_{0}^{\infty}\left(\int_{x}^{x+h}|\psi(u,\chi)|e^{-\frac{u}{N}}du\right)^2dx.
\end{split}
\]
%The second term here is 
%\begin{align*}
%\frac{h^2}{N^2}\int_{0}^{\infty}|\psi(x+h,\chi)|^2e^{-\frac{2x}{N}}dx &=\frac{h^2}{N^2}\int_{h}^{\infty}|\psi(x,\chi)|^2e^{-\frac{2(x-h)}{N}}dx\\
%&=\frac{h^2}{N^2}e^{\frac{2h}{N}}\int_{h}^{\infty}|\psi(x,\chi)|^2e^{-\frac{2x}{N}}dx\\&\le \frac{e^2h^2}{N^2}\int_{0}^{\infty}|\psi(x,\chi)|^2e^{-\frac{2x}{N}}dx,
%\end{align*}
%and
For the third term we use the Cauchy--Schwarz inequality and change the order of integration to see
\begin{align*}
\frac{1}{N^2}\int_{0}^{\infty}\left(\int_{x}^{x+h}|\psi(u,\chi)|e^{-\frac{u}{N}}du\right)^2dx 
&\le \frac{h}{N^2}\int_{0}^{\infty}\left(\int_{x}^{x+h}|\psi(u,\chi)|^2e^{-\frac{2u}{N}}du\right)dx\\
&=\frac{h}{N^2}\int_{0}^{\infty}\left(\int_{u-h}^{u}dx\right)|\psi(u,\chi)|^2e^{-\frac{2u}{N}}\, du \\& =\frac{h^2}{N^2}\int_{0}^{\infty}|\psi(u,\chi)|^2e^{-\frac{2u}{N}}du.
\end{align*}
Hence we have proved
\begin{align}\label{goma20}
I_2(N,h) \ll \frac{h^2}{N^2}\int_{0}^{\infty}|\psi(x,\chi)|^2e^{-\frac{2x}{N}}dx
+ \int_{0}^{\infty}|\psi(x+h,\chi)-\psi(x,\chi)|^2e^{-\frac{2x}{N}}dx.
\end{align}

Finally, since
\begin{align*}
\int_{0}^{\infty}f(x)e^{-\frac{2x}{N}}dx
&=\int_{0}^{N}f(x)e^{-\frac{2x}{N}}dx+\sum_{j=1}^{\infty}\int_{jN}^{(j+1)N}f(x)e^{-\frac{2x}{N}}dx\\
&\le \int_{0}^{N}f(x)e^{-\frac{2x}{N}}dx+\sum_{j=1}^{\infty}e^{-2j}\int_{jN}^{(j+1)N}f(x)dx\\
&\le \sum_{j=1}^{\infty}\frac{1}{2^{j-1}}\int_{0}^{jN}f(x)dx
\end{align*}
for any function $0\le f(x)\ll |x|^k$ (for some $k\in\mathbb{N}$) on $[0,\infty)$, we have 
\begin{align}\nonumber
I_2(N,h) &\ll \sum_{j=1}^{\infty}\frac{1}{2^j}\left(\frac{h^2}{N^2}\int_{0}^{jN}|\psi(x,\chi)|^2dx+\int_{0}^{jN}|\psi(x+h,\chi)-\psi(x,\chi)|^2dx\right)\\\label{goma21}
&\ll \sum_{j=1}^{\infty}\frac{1}{2^j}\left(\frac{h^2}{N^2}J_1(jN)+J_2(jN,h)\right).
\end{align}
In what follows we will by \eqref{goma12} always choose $h=\frac{N}{2^{k+2}}$ for $0\le k < \log_2\!N$ so that $\frac14\le h \le \frac{N}{4}$. By \eqref{goma13}, \eqref{goma15}, \eqref{goma21}, and \cref{GVlemma}, we have on using $2\le q\le N$ that 
\begin{align*}
&\int_{0}^{\frac{1}{2h}}|\Psi(z,\chi)|^2d\alpha\\
&\qquad\ll\frac{1}{h^2}\int_{0}^{h}\left|\sum_{n\le x}\chi(n)\Lambda(n)e^{-\frac{n}{N}}\right|^2dx+\frac{1}{h^2}\int_{0}^{\infty}\left|\sum_{x<n\le x+h}\chi(n)\Lambda(n)e^{-\frac{n}{N}}\right|^2dx\\
&\qquad=\frac{1}{h^2}\left(I_1(N,h)+I_2(N,h)\right)\\
&\qquad\ll\frac{1}{h^2}\left(J_1(h)+\sum_{j=1}^{\infty}\frac{1}{2^j}\left(\frac{h^2}{N^2}J_1(jN)+J_2(jN,h)\right)\right)\\
&\qquad\ll\frac{1}{h^2}(h^2\log^2\!q)+\frac{1}{N^2}\sum_{j=1}^{\infty}\frac{1}{2^j}(jN)^2\log^2\!q+\frac{1}{h^2}\sum_{j=1}^{\infty}\frac{1}{2^j}jNh\log^2\left(\frac{3qjN}{h}\right)\\
&\qquad\ll \left(1+\sum_{j=1}^{\infty}\frac{j^2}{2^j}\right)\log^2\!q+\frac{N}{h}\log^2\left(\frac{3qN}{h}\right)\sum_{j=1}^{\infty}\frac{j}{2^j}+\frac{N}{h}\sum_{j=1}^{\infty}\frac{j\log^2\!j}{2^j}\\
%&=7\log^2\!q+\frac{2N}{h}\log\left(\frac{3qN}{h}\right)+\frac{N}{h}\sum_{j=1}^{\infty}\frac{j}{2^j}\log j\\
&\qquad\ll \frac{N}{h}\log^2\!N.
\end{align*}
We thus obtain by \eqref{goma12} that
\begin{align}\nonumber
\mathcal{I}_2
&\ll N\sum_{0\le k<\log_2\!N}\frac{1}{2^k} \left(\underset{\chi\ne\chi_0}{\rm max}\int_{0}^{\frac{2^{k+1}}{N}}|\Psi(z,\chi)|^2d\alpha\right) \\\nonumber
&\ll N\sum_{0\le k<\log_2\!N}\frac{1}{2^k}\left(\frac{N}{N/2^{k+2}}\log^2\!N\right)\\\nonumber
&=N\log^2\!N\sum_{0\le k<\log_2\!N}\frac{2^{k+2}}{2^k}\\\nonumber
&=4N\log^2\!N\sum_{0\le k<\log_2\!N}1 \\
%&\le 4N\log^2\!N\frac{\log N}{\log2} \nonumber\\
&\ll N\log^3\!N. \label{goma23}
\end{align}

\medskip
%-----------Part3------------------------------
From \eqref{OnlyI2left} and \eqref{goma23}, we conclude, for $2\le q\le N$,
\begin{align*}G_q(N)
&=\frac{G(N)}{\phi(q)}+\mathcal{O}(N\log^3\!N)+\mathcal{O}\left(\frac{N}{\phi(q)}\log^3\!N\right)+\mathcal{O}(\log^5\!N)\\
&=\frac{G(N)}{\phi(q)}+\mathcal{O}(N\log^3\!N).
\end{align*}
This completes the proof of \cref{main}.
\qed

%-----------2.2------------------------------------------------------
%\subsection{Proof of \cref{cor}}

%----------------------Acknowledgements--------------------------------------------------
\section*{Acknowledgement}
The authors thank Professor Daniel Alan Goldston and Professor Masanobu Kaneko for their valuable suggestions and careful reading of the draft. The authors also thank Professor Yuta Suzuki and Professor Kohji Matsumoto for their helpful comments and remarks, and Professor Alessandro Languasco for additional remarks.
The first author was supported by the WISE program (MEXT) at Kyushu University. The second author was supported by JSPS KAKENHI Grant Number 22K13895.

%----------------参考文献---------------------------------

%\bibliographystyle{amsplain}%{amsalpha}
%\bibliography{Reference} 

% --------------------------------------------------------------------------
\end{document}